\documentclass[12pt]{amsart}

\def\update{13/11/2012}

\usepackage{amsmath}
\usepackage{amsthm}
\usepackage{amssymb}
 \theoremstyle{plain}

\newtheorem{proposition}{Proposition}[section]

\theoremstyle{remark}
\numberwithin{equation}{section}
\newtheorem{theorem}{Theorem}[section]

\def\atop#1#2{
\genfrac{}{}{0pt} {}
{#1}
{#2}}

\begin{document}

\title{Matrices of finite abelian groups,  \\
Finite Fourier Transform and codes}
\author{S. Kanemitsu and M. Waldschmidt}
\date{\update}

\begin{abstract}
Finite (or Discrete) Fourier Transforms (FFT) are essential tools in engineering disciplines based on signal transmission, which is the case in most of them. FFT are related with circulant matrices, which can be viewed as group matrices of cyclic groups.
In this regard, we introduce a generalization of the previous investigations to the case of finite groups, abelian or not. We make clear the points which were not recognized as underlying algebraic structures. Especially, all that appears in the FFT in engineering has been elucidated from the point of view of linear representations of finite groups. We include many worked-out examples for the readers in engineering disciplines.
\end{abstract}

\maketitle

\section{The matrix of a finite abelian group }\label{S:MatrixFiniteAbelianGroup}

\subsection{Matrix of a finite group}

Let $G$ be a finite group of order $n$, and let $F$ be a field of characteristic not dividing $n$. This setting will be used throughout;  we also assume that $F$ contains a primitive $n$-th root
$\zeta=\zeta_n$ of $1$, but sometimes we will consider subfields of $F$ which do not satisfy this condition. We use the symbols $j=\zeta_3=e^{2\pi i/3}$ and $i=\zeta_4=e^{2\pi i/4}$ to mean a primitive cube and fourth root of unity, the latter expressions valid in characteristic $0$.

    Let $\underline{X}:=(X_\sigma)_{\sigma\in G}$ be an $n$-tuple of variables indexed by the elements of $G$. The {\it group matrix}\index{group matrix}
$$
A_G:=
\bigl(X_{\tau^{-1}\sigma}
\bigr)_{\tau,\sigma\in G} \in {\mathrm {Mat}}_{n\times n}\bigl( F[\underline{X}] \bigr),
$$
(depending on a labeling of elements of $G$) has been introduced by Dedekind in the course of his investigation on normal bases for Galois extensions. In 1886, Frobenius gave a complete factorization of the determinant of $A_G$ into irreducible factors in $F[\underline{X}]$ -- this was the start of the theory of linear representations and characters of finite groups.

\subsection{Matrix of a finite abelian group }

We will consider the general case of a finite group in \S$\ref{S:MatrixFiniteGroup}$; here we assume the group $G$ to be  abelian and we take $F={\mathbf C}$. Let $\widehat{G}$  be the dual of $G$, which is the group ${\mathrm {Hom}}(G,{\mathbf C}^\times)$ of characters of $G$. We will consider $n$--tuples of complex numbers; when they are indexed by the elements of $G$, we say that they are in ${\mathbf C}^G$; when they are indexed by the elements of $\widehat{G}$, we say that they are in ${\mathbf C}^{\widehat{G}}$.
\bigskip

For each $\chi\in  \widehat{G}$, the vector
$$
 \bigl(
\chi(\sigma)
\bigr)_{\sigma\in G}\in{\mathbf C}^G
$$
is an eigenvector of the matrix  $A_G$ belonging to the eigenvalue given by the linear form
\begin{equation}\label{eigenvalue1}
Y_\chi:=\sum_{\sigma\in G} \chi(\sigma) X_\sigma.
\end{equation}
This follows from the relation, for $\chi\in\widehat{G}$,
$$
\sum_{\sigma\in G} \chi(\sigma) X_{\tau^{-1}\sigma}  =  \chi (\tau) \sum_{\sigma\in G} \chi (\sigma) X_\sigma.
$$
Therefore the $n\times n$ matrix
\begin{equation}\label{Equation:P}
P:= \bigl(
\chi(\sigma)
\bigr)_{\sigma \in G, \chi\in \widehat{G}} \in {\mathrm {Mat}}_{n\times n} ( {\mathbf C})
\end{equation}
is regular and
\begin{equation}\label{Equation:AGP=PD}
A_GP=PD,
\end{equation}
where $D$ is the diagonal $n\times n$ matrix
$$
D:={\mathrm {Diag}}(Y_\chi)_{\chi\in \widehat{G}}:= \bigl(
Y_\chi \delta_{\chi, \psi}
\bigr)_{ \chi, \psi  \in \widehat{G}}.
$$
We have used Kronecker's symbol
$$
\delta_{\chi, \psi}=\begin{cases}
1& \hbox{if $\chi = \psi$},
\\
0& \hbox{if $\chi\not=\psi$}.
\end{cases}
$$
In particular the determinant of $A_G$, called  {\it  determinant of the group }\index{group determinant} $G$  ({\it Gruppendeterminant} in German -- see the historical note of \cite{Bki}), is
\begin{equation}\label{Gruppendeterminant}
\det A_G=\prod_{\chi\in \widehat{G}} Y_\chi= \prod_{\chi\in \widehat{G}}\sum_{\sigma\in G} \chi(\sigma)   X_\sigma.
\end{equation}
This formula is used by Hasse \cite{H} to give an explicit formula for the number of ideal classes of an algebraic number field (see also \cite{Yam}). It is also useful for computing the $p$--adic rank of the units of an algebraic number field \cite{Ax}.

\bigskip

The dual $\widehat{\widehat{G}}$ of $\widehat{G}$, which is called the {\it bidual}\index{bidual} of $G$, is canonically isomorphic to $G$, the characters of $\widehat{G}$ being given by $\chi\mapsto \chi(\sigma)$ for $\sigma\in G$. Denoting by ${\mathbf U}_n$ the group of $n$--th roots of unity in ${\mathbf C}$, namely the set of complex roots of the polynomial $X^n-1$, the pairing 
$$
\begin{matrix}
G\times  \widehat{G}&\longrightarrow & {\mathbf U}_n
\\
(\sigma,\chi)&\longmapsto&\chi(\sigma)\\
\end{matrix}
$$
is non--degenerate. 
\bigskip

In parallel to the case of the dual of $G$, we introduce their counterparts. Correspondence can be seen in the table below.

Let $\underline{T}:=(T_\chi)_{\chi\in \widehat{G}}$ be an $n$-tuple of variables indexed by $\widehat{G}$.
The matrix $A_{\widehat{G}}\in {\mathbf C}[\underline{X}] $ of the dual of $G$ is: 
$$
A_{\widehat{G}} =
\bigl(T_{\psi^{-1}\chi}
\bigr)_{\psi,\chi\in G} \in {\mathrm {Mat}}_{n\times n}\bigl( {\mathbf C}[\underline{X}] \bigr).
$$
For each $\sigma\in G$, the vector
$$
 \bigl(
\chi(\sigma)
\bigr)_{\chi \in \widehat{G}}\in{\mathbf C}^{\widehat{G}}
$$
is an eigenvector of the matrix  $A_{\widehat{G}}$ belonging to the eigenvalue given by the linear form
\begin{equation}\tag{$\ref{eigenvalue1}$'}
U_\sigma:=\sum_{\chi\in \widehat{G}} \chi(\sigma) T_\chi.
\end{equation}
This follows from the relation, for $\sigma\in G$,
$$
\sum_{\chi\in \widehat{G}} \chi(\sigma) T_{\psi^{-1}\chi}  =  \psi(\sigma) \sum_{\chi\in \widehat{G} } \chi (\sigma) T_\chi.
$$
Therefore the transpose ${^t}\! P$ of the matrix $P$ given by ($\ref{Equation:P}$), namely
\begin{equation}\tag{$\ref{Equation:P}$'}
{^t}\! P:= \bigl(
\chi(\sigma)
\bigr)_{\chi\in \widehat{G}, \sigma \in G} \in {\mathrm {Mat}}_{n\times n} ( {\mathbf C})
\end{equation}
satisfies
\begin{equation}\tag{$\ref{Equation:AGP=PD}$'}
A_{\widehat{G}} {^t}\! P=^{t}\! \! \!  P \widehat{D},
\end{equation}
where $\widehat{D}$ is the diagonal $n\times n$ matrix
$$
\widehat{D}:={\mathrm {Diag}}(U_\sigma)_{\sigma\in \widehat{G}}.
$$
\bigskip

\begin{center}
Table. Correspondence between $G$ and $\widehat{G}$
\end{center}
\begin{center}
\begin{tabular}{|r|l|c|c|c|} \hline
group  & $n$-tuples  & eigenvalues & vectors   \\ \hline
 $G$ & $X_{\sigma}$ & $Y_{\chi}$ & ${\chi(\sigma)}_{\sigma \in G}$  \\ \hline
$\widehat{G}$ & $T_{\chi}$ & $U_{\sigma}$ & ${\chi(\sigma)}_{\chi \in \widehat{G}}$  \\ \hline
\end{tabular}
\end{center}

\subsection{Matrix of a cyclic group}

We consider here the special case where the group $G$ is the cyclic group $C_n$ of order $n$. Let   $\sigma_1$ be a generator of $G$ and  $\chi_1$  a generator of the cyclic group $\widehat{G}$. Then the number $\zeta=\chi_1(\sigma_1)$ is a primitive $n$--th root of unity which we have assumed to belong to $F$. We have
$G=\{1,\sigma_1,\dots,\sigma_1^{n-1}\}$,   $\widehat{G}=\{1,\chi_1,\dots,\chi_1^{n-1}\}$ and we set  $X_i=X_{\sigma_1^i}$ and  $Y_{\ell}=Y_{\sigma_1^\ell}$. Then
$$
A_{C_n}=
\begin{pmatrix}
X_0& X_1&X_2&\cdots &X_{n-1}
\\
X_{n-1}& X_0&X_1&\cdots &X_{n-2}
\\
\vdots&\vdots&\vdots&\ddots&\vdots
\\
X_{2}& X_3&X_4&\cdots &X_{1}
\\
X_{1}& X_2&X_3&\cdots &X_{0}
\\
\end{pmatrix}
$$
is a circulant (see \cite{Dav}) and may be expressed as $\sum_{\ell=0}^{n-1}X_{\ell}K^{\ell}$ (cf. \S 1.4), where $K$ is the $n\times n$ matrix which is the specialization of $A_{C_n}$ at 
$$
(X_0,X_1,X_2,\dots,X_{n-1})=(0,1,0,\dots,0).
$$

Since $\chi_i(\sigma_1^{\ell})=\zeta^{i\ell}$, \eqref{eigenvalue1} reads
$$
Y_{\ell}=\sum_{i=0}^{n-1} \zeta^{i\ell}X_i.
$$
The matrix $P$ is
$$
P= \bigl(\zeta^{ij}\bigr)_{0\le j,j\le n-1}=
\begin{pmatrix}
1& 1 & 1 &\cdots &1
\\
1& \zeta &\zeta^2 &\cdots &\zeta^{n-1}
\\
1&\zeta^2&\zeta^4&\cdots &\zeta^{2(n-1)}
\\
\vdots&\vdots&\vdots&\ddots&\vdots
\\
1&\zeta^{n-1} & \zeta^{2(n-1)}&\cdots &\zeta^{(n-1)(n-1)}
\\
\end{pmatrix},
$$
where the exponent of $\zeta$ is given by the multiplication table of the ring ${\mathbf Z}/n{\mathbf Z}$. The determinant $\Delta_n$ of $P$ is considered by Massey in \cite{Mas}. It is a Vandermonde determinant, with its value
$$
\Delta_n=\prod_{\ell=1}^{n-1}\prod_{i=0}^{\ell-1} \zeta^i (\zeta^{\ell-i}-1).
$$
For instance
$$
\Delta_1=1,\quad
\Delta_2=
\det\begin{pmatrix}
1&1
\\
1&-1
\end{pmatrix}=-2,
$$
$$
\Delta_3=
\det\begin{pmatrix}
1&1&1
\\
1& j & j^2
\\
1& j^2& j
\end{pmatrix}=
3j(j-1),
$$
$$
\Delta_4=\det\begin{pmatrix}
1&1&1&1
\\
1&i&-1&-i
\\
1&-1&1&-1
\\
1&-i&-1&i
\end{pmatrix}=
16i,
$$
where $j=\zeta_3$ and $i=\zeta_4$ are the primitive third root and fourth root of $1$, respectively.
 In general, the sum of the $n$ rows is
 $$
 (n,0,0,\dots 0),
 $$
 and the determinant of $P$ is $n$ times the determinant $\Delta'_n$  of the $(n-1)\times (n-1)$ matrix $\bigl(\zeta^{ij}\bigr)_{1\le i, j \le n-1}$. If $n$ is prime, after a suitable permutation of the rows, one can write $\Delta'_n$ as a circulant determinant with first row $(\zeta, \zeta^2,\dots,\zeta^{n-1})$.

 There are various subfields of $F$ over which one can decompose the group determinant into a product of irreducible factors.

 Firstly, over $F$ itself (which contains the $n$-th roots of unity), the decomposition is given by ($\ref{Gruppendeterminant}$).

Secondly, in characteristic zero, over $\mathbf Q$, the polynomial $X^n-1$ splits as
\begin{equation}\label{Equation:cyclotomie}
X^n-1=\prod_{d\mid n} \Phi_d(X),
\end{equation}
where $\Phi_d$ is the {\it cyclotomic polynomial}\index{cyclotomic polynomial} of index $d$, which is an irreducible polynomial in ${\mathbf Z}[X]$  of  degree $\varphi(d)$. Let $\zeta_d$ be a root of $\Phi_d$ (i.e. a primitive $d$--th root of unity). Then it generates
 the $d$-th {\it cyclotomic field}\index{cyclotomic field} over ${\mathbf Q}$
 \begin{equation}\label{cyclotomicfield1}
\Gamma_d:={\mathbf Q}[X]/(\Phi_d(X))={\mathbf Q}(\zeta_d).
\end{equation}

Accordingly, the group determinant  ($\ref{Gruppendeterminant}$) of a cyclic group of order $n$  splits into a product of irreducible polynomials in ${\mathbf Q}[\underline{X}]$
$$
\det A_G=\prod_{d\mid n}  \psi_d(\underline{X}),
$$
where the homogeneous polynomial $\psi_d\in{\mathbf Q}[\underline{X}]$ is given by the norm ${\mathrm {N}}_{\Gamma_d/{\mathbf Q}}$ of $\Gamma_d$

$$
\psi_d(\underline{X})={\mathrm {N}}_{\Gamma_d/{\mathbf Q}}(X_0+\zeta_d X_1+\cdots+\zeta_d^{n-1}X^{n-1}).
$$
For instance, for $n=3$, the determinant of the cyclic group $C_3$ of order $3$ is
$$
\det A_{C_3}=\left|
\begin{matrix}
X_0&X_1&X_2
\\
X_2&X_0&X_1
\\
X_1&X_2&X_0
\end{matrix}
\right| = X_0^3+X_1^3+X_2^3-3X_0X_1X_2,
$$
over ${\mathbf C}$ the decomposition  ($\ref{Gruppendeterminant}$) is
 \begin{equation}\label{Equation:troissurC}
 (X_0+X_1+X_2)(X_0+jX_1+j^2X_2)(X_0+j^2X_1+jX_2),
 \end{equation}
 while over ${\mathbf Q}$ the decomposition is
 \begin{equation}\label{Equation:troissurQ}
 (X_0+X_1+X_2)(X_0^2+X_1^2+X_2^2- X_0X_1-X_1X_2-X_2X_0),
\end{equation}
 where the second factor is 
 $$
 {\mathrm {N}}_{{\mathbf Q}(j)/{\mathbf Q}}(X_0+jX_1+j^2X_2).
 $$

Thirdly, in finite characteristic,  over a finite field ${\mathbf F}_q$ with $q$ elements (and $\gcd(q,n)=1$), the decomposition of the  determinant  of the cyclic group $C_n$ is given by the decomposition of the cyclotomic polynomials $\Phi_d$, with $d$ ranging over the set of divisors of $n$, over ${\mathbf F}_q$. For such a $d$, let $r$ be the order of $q$ in the multiplicative group ${({\mathbf Z}/d{\mathbf Z})}^\times$. Then $\Phi_d$ splits in ${\mathbf F}_q[X]$ into $\varphi(d)/r$ polynomials, all of the same degree $r$. If $H$ is the subgroup generated by the class $q$ modulo $d$ in ${({\mathbf Z}/d{\mathbf Z})}^\times$, the choice of a primitive $d$-th root of unity $\zeta_d$ gives rise to an irreducible factor
$$
P_H(X)=\prod_{h\in H} (X-\zeta_d^h),
$$
and all factors of $\Phi_d$ are obtained by taking the $\varphi(d)/r$ classes of ${({\mathbf Z}/d{\mathbf Z})}^\times$ modulo $H$; for any $m\in  {({\mathbf Z}/d{\mathbf Z})}^\times$, set
$$
P_{mH}(X)=\prod_{h\in H} (X-\zeta_d^{mh}).
$$
Then the decomposition of $\Phi_d$ into irreducible factors over ${\mathbf F}_q$ is
$$
\Phi_d(X)=\prod_{mH\in {({\mathbf Z}/d{\mathbf Z})}^\times/H} P_{mH}(X).
$$
Here is another description of the decomposition of the polynomial $X^n-1$ into irreducible factors over ${\mathbf F}_q$.
The $2^n$ factors of the polynomial $X^n-1$ over a field containing a primitive $n$--th root of unity $\zeta$ are
$$
Q_L(X)=\prod_{\ell \in L} (X-\zeta^\ell), \quad (L\subseteq \{1,\dots,n\})
$$
(with $Q_{\emptyset}=1$, as usual),
and such a polynomial belongs to ${\mathbf F}_q[X]$ if and only if $Q_L(X)^q=Q_L(X^q)$. This condition is satisfied if and only if the label set $L$ of $ \{1,\dots,n\}$ is stable under multiplication by $q$ in ${\mathbf Z}/n{\mathbf Z}$. Hence the irreducible factors of $X^n-1$ over ${\mathbf F}_q$ are the $Q_L$ with $L$ stable under multiplication by $q$ and minimal for this property.

Once we know the decomposition of $X^n-1$, we deduce the decomposition of the determinant of the cyclic group $G$ of order $n$.

\noindent
{\tt Example}.
Consider the cyclic group $C_3$ of order $3$, assuming that the characteristic is not $3$.  For $q\equiv 1 \mod 3$, the polynomial $X^3-1$ has the decomposition ($\ref{Equation:troissurC}$) with three homogeneous linear factors (because  ${\mathbf F}_q$ contains the primitive cubic roots of $1$), while for $q\equiv 2 \mod 3$, the polynomial $X^3-1$ has the decomposition ($\ref{Equation:troissurQ}$) with one homogeneous linear factor and one irreducible factor of degree $2$ (because $X^2+X+1$ is irreducible over  ${\mathbf F}_q$).

\subsection{The group ring of a cyclic group and the algebra of circulants}
Recall that $F$ is a field whose characteristic does not divide $n$. Let $K$ denote the $n\times n$  circulant matrix with its first row $(0,1,0,\dots,0)$ (often referred to as the {\it shift-forward matrix}\index{shift-forward matrix}), where a {\it circulant matrix}\index{circulant matrix} is one whose rows consists of the $n$  cycles 
$$
(c_0,\dots,c_{n-1}), \; (c_{n-1},c_0,\dots,c_{n-2}), \;  \dots ,
$$
which therefore
can be written as 
$$
\begin{pmatrix}
c_0& c_1&c_2&\cdots &c_{n-1}
\\
c_{n-1}& c_0&c_1&\cdots &c_{n-2}
\\
\vdots&\vdots&\vdots&\ddots&\vdots
\\
c_{2}& c_3&c_4&\cdots &c_{1}
\\
c_{1}& c_2&c_3&\cdots &c_{0}
\\
\end{pmatrix}
=
c_0I+c_1K+\cdots+c_{n-1}K^{n-1}, 
$$
so that the algebra of circulant $n\times n$ matrices is nothing other than $F[K]$, which we denote by $R$ subsequently. Further, the minimal polynomial of $K$ is $T^n - 1$. Hence $F[K]$ is isomorphic to $ F[T]/(T^n-1)$.

If $C_n$ denotes a cyclic group of order $n$, then the algebra $F[C_n]$, called the {\it group ring}\index{group ring} of $C_n$ over $F$,  is also isomorphic to $ F[T]/(T^n-1)$. Altogether,

\begin{equation}
R=F[K]\simeq F[T]/(T^n-1) \simeq F[C_n].
\end{equation}

Assume $F$ contains a primitive $n$--th root $\zeta$ of unity. We split the polynomial $T^n-1$ into irreducible factors over  $F$, say
$$
T^n-1=\prod_{\ell=0}^{n-1} (T-\zeta^\ell).
$$
Then the algebra $F[C_n]$ splits accordingly into a product of $n$ algebras, all isomorphic to $F$:
\begin{equation}\label{equation:isomorphisme}
F [C_n]\simeq  \prod_{\ell=0}^{n-1} F[T]/(T-\zeta^\ell).
\end{equation}
For $0\le \ell\le n-1$, denote by $R_\ell$ the subset of $ R$ which is  the image of the factor $ F[T]/(T-\zeta^\ell)$ ($0\le \ell\le n-1)$  on the right hand side of ($\ref{equation:isomorphisme}$). Then $R_\ell$   is a simple $F$--  algebra  and $R=R_0\times\cdots\times R_{n-1}$. Let $E_\ell$ be the unity element of $R_\ell$. Then we have the decomposition into orthogonal idempotents
$$
R_\ell=R E_\ell, \quad
 1=E_0+\cdots+E_{n-1}
 \quad \hbox{and}\quad E_iE_\ell=\delta_{i\ell}
 \quad (0\le i, \ell\le n-1).
 $$
For  the structure theorem of semi--simple rings, see for instance  \cite{La}, Th.~4.4 in Chap.~XVII or  \cite{Bki}.
The special case  of the algebra  $R=F[K]$ of circulants  of order $n$  is worked out in \cite{Wil}:  the solution is
\begin{equation}\label{Equation:Wilkie}
E_h=\frac{1}{n} \sum_{\ell=0}^{n-1} \zeta^{-h\ell} K^\ell, \quad(0\le h\le n-1).
\end{equation}
In the other direction we have
$$
K^h=\sum_{\ell=0}^{n-1} \zeta^{h\ell} E_\ell, \quad(0\le h\le n-1).
$$

These formulae are easy to check, but it is interesting to explain where they come from. The isomorphism ($\ref{equation:isomorphisme}$) from $F[G]$ to the product of the algebras $F[T]/(T-\zeta^\ell)$ maps the class modulo $X^n-1$ of a polynomial $P$ to the $n$--tuple $(P(\zeta^\ell))_{0\le \ell\le n-1}$. We want to explicitly write down the inverse isomorphism. Given an $n$--tuple $(b_\ell)_{0\le \ell\le n-1}$, one deduces from the Chinese Remainder Theorem that there is a unique polynomial $P$ of degree $\le n-1$ such that $P(\zeta^h)=b_h$ for $0\le h \le n-1$. To write down the solution $P$ amounts to solving the associated interpolation problem, which is done by classical interpolation formulae. In this specific case, they lead us to introducing the polynomial
$$
P_0(X)=\frac{1}{n}\cdot \frac{X^n-1}{X-1}=
\frac{1}{n}(X^{n-1}+\cdots+X+1).
$$
It satisfies $P_0(1)=1$ and $P_1(\eta)=0$ for any $n$--th root of unity $\eta$ not equal to $1$. Hence for $0\le h\le n-1$,   the polynomial $P_h(X):=P_0(X/\zeta^\ell)$, which is
$$
P_h(X)=\frac{1}{n} \sum_{\ell=0}^{n-1} \zeta^{-h\ell} X^\ell,
$$
satisfies
$$
P_h(\zeta^\ell)=\delta_{h,\ell} \quad \hbox{for $0\le  h,\ell\le n-1$.}
$$
This is how ($\ref{Equation:Wilkie}$) arises: $E_h=P_h(K)$.
Also, the solution of the interpolation problem is therefore the following: the polynomial
$$
P(X)=\sum_{h=0}^{n-1} b_h P_h(X)
$$
satisfies $P(\zeta_h)=b_h$ for  $0\le h \le n-1$.

In characteristic $0$, there is another basis for the circulant algebra, which is rational over ${\mathbf Q}$.  Let $n$ be a positive integer. We consider the   decomposition, into a product of simple algebras over ${\mathbf Q}$, of the semi--simple algebra ${\mathbf Q}[X]/(X^n-1)$
associated with the decomposition ($\ref{Equation:cyclotomie}$) of the polynomial $X^n-1$ into irreducible factors over ${\mathbf Q}$:
$$
{\mathbf Q}[X]/(X^n-1)=\prod_{d\mid n}  \Gamma_d,
$$
where $\Gamma_d$ is the $d$--th cyclotomic field defined by \eqref{cyclotomicfield1}.
For each  divisor $d$ of $n$, define
$$
\Psi_{n,d}(X)=\frac{X^n-1}{\Phi_d(X)}=\prod_{\atop{d'\mid n}{ d'\neq d}} \Phi_{d'}(X).
$$
Since $\Phi_d$ and $\Psi_{n,d}$ are relatively prime,   there is a unique polynomial $\widetilde{\Psi}_{n,d}$ of degree $\le \varphi(d)-1$ which is the inverse of  $\Psi_{n,d}$ modulo $\Phi_d$:
$$
\widetilde{\Psi}_{n,d}\Psi_{n,d}\equiv 1 \mod {\Phi_d}.
$$
Then  a basis of the ${\mathbf Q}$--algebra ${\mathbf Q}[C_n]$ is given by
\begin{equation}\label{Equation:KofGcyclic}
\bigl\{E_{d,j} \; \mid \; 0\le j\le \varphi(d)-1,\; d\mid n\bigr\},
\end{equation}
where
$$
 E_{d,j}  \equiv  X^j \widetilde{\Psi}_{n,d}(X) \Psi_{n,d} \mod  (X^n-1).
$$

As an example, consider the case where  $n=p$ is a prime. We have
$$
\Psi_{p,1}=\frac{X^p-1}{X-1}=\Phi_p,\quad
\Psi_{p,p}=\frac{X^p-1}{\Phi_p}=X-1=\Phi_1,
$$
hence $ \widetilde{\Psi}_{p,1}=1/p$. To compute $ \widetilde{\Psi}_{p,p}$, we  start by taking the derivative of
$X^p-1=(X-1)\Phi_p$:
$$
pX^{p-1} =\Phi_p(X)+(X-1)\Phi'_p(X).
$$
Hence the polynomial
$$
\widetilde{\Psi}_{p,p}:=\frac{1}{p} \Phi'_p-  \frac{X^{p-1}-1}{X-1}
$$
satisfies
$$
(X-1)\widetilde{\Psi}_{p,p} =1-\frac{1}{p} \Phi_p.
$$
Therefore a basis of the circulant algebra with $n=p$  is  given by
$((1/p)\Phi_p, F_0, F_1 ,\ldots, F_{p-2})$, with
$$
F_j\equiv   X^{j} (X-1)   \widetilde{\Psi}_{p,1}(X)  \mod (X^p-1)
\quad(0\le j\le p-2).
$$

For instance, when $p=3$, we have
$$
 \widetilde{\Psi}_{3,1}(X)=-\frac{1}{3}(X+2), \quad
 (X-1)\widetilde{\Psi}_{3,1}(X) = -\frac{1}{3}(X+2)(X-1),
 $$
 $$
 (X-1)X\widetilde{\Psi}_{3,1}(X) = -\frac{1}{3}(X+2)(X-1)X\equiv  -\frac{1}{3}(X-1)^2\mod (X^3-1),
$$
and  the basis of ${\mathbf Q}[X]/(X^3-1)$ which is associated to the basis $(1,0)$, $(0,1)$, $(0,X)$ of the product
$$
{\mathbf Q}[X]/(X-1)\times{\mathbf Q}[X]/(X^2+X+1)
$$
under the natural isomorphism is given by the classes modulo $X^3-1$ of the polynomials
$$
\frac{1}{3}(X^2+X+1),\quad
-\frac {1}{3}(X^2+X-2),\quad
-\frac{1}{3}(X^2-2X+1).
$$

\noindent
{\tt Remark.}
There is no element   $J$ in the algebra ${\mathbf Q}[X]/(X^3-1)$ which satisfies $1+J+J^2=0$. This is analogous to the fact that the product algebra ${\mathbf Q}\times {\mathbf Q}[i]$ does not contain a square root of $-1$. In a product $A_1\times A_2$ of two algebras, there are in general no subalgebras isomorphic to the factors $A_1$ and $A_2$.

\subsection{The group ring $F[G]$ of a finite abelian group $G$}

We extend the results of the previous section to the  algebra $F[G]$ of a finite abelian group $G$. Here we assume that $F^\times$ contains a subgroup of order $n$, where $n$ is the order of $G$. Hence the characteristic of $F$ does not divide $n$.

According to Maschke's Theorem (see for instance \cite{La} Chap.~XVIII,  \S1, Th.~1.2; see also \cite{Se} Chap.~6, Prop.~9 for the characteristic zero case), the algebra $F[G]$ is semi--simple: it is a product of $n$ algebras isomorphic to $F$. Under such an isomorphism, the canonical basis of $F^G$ is associated with a basis $(e_\chi)_{\chi\in \widehat{G}}$ of $F[G]$  satisfying
$$
e_\chi e_\psi = \delta_{\chi,\psi} e_\chi.
$$
An explicit solution is given by
$$
e_\chi=\frac{1}{n} \sum_{\sigma \in G} \chi^{-1}(\sigma)\sigma
\quad (\chi \in \widehat{G}).
$$
This follows from the  relation of orthogonality of characters   (see \cite{Se}, Th.~3, \S2.3 and \cite{La} Chap.~XVIII, \S5,  Th.~5.1):
$$
\frac{1}{n}  \sum_{\tau\in G} \chi (\tau) \psi^{-1}(\tau ) =
\delta_{\chi,\psi}.
$$

One obtains a basis of ${\mathbf Q}[G]$ rational over ${\mathbf Q}$ by writing the group $G$ as a product of cyclic groups $C_{d_1}\times \cdots\times C_{d_k}$ of orders $d_1,\cdots,d_k$ respectively, with $d_1\mid d_2\mid \cdots \mid d_k$, where $d_1,\dots,d_k$ are the elementary divisors of the finitely generated ${\mathbf Z}$--module $G$ (for the elementary divisors theorem, see for instance  \cite{La} Chap.~III, Th.~7.8). Then each algebra ${\mathbf Q}[C_{d_i}]$ has a rational basis given by ($\ref{Equation:KofGcyclic}$), and one deduces a rational basis for the product   ${\mathbf Q}[C_{d_1}]\times \cdots \times {\mathbf Q}[C_{d_k}]={\mathbf Q}[G]$. The norm of the generic element gives the decomposition
into irreducible factors of the group determinant over ${\mathbf Q}$.

\section{Finite Fourier Transform associated with a finite abelian group}

\subsection{Generalized Finite Fourier Transform}

We keep the notation introduced in \S$\ref{S:MatrixFiniteAbelianGroup}$,
 with the field $F$, the finite abelian group $G$ with $n$ elements, with $n$ being relatively prime to the characteristic of $F$. An element $\underline{b}$ in $F^G$ is an $n$--tuple of elements of $F$ indexed by $G$, and
 an element $\underline{B}$ in $F^{\widehat{G}}$ is an $n$--tuple of elements of $F$ indexed by $\widehat{G}$.
The following proposition gives a Finite Fourier Transform Pair for $G$. For Finite Fourier Transforms on a more general finite group, cf. e.g. \cite{Aus} and \cite{Diaconis}.

\begin{proposition} 
\label{Proposition:InverseTransform}
For $\underline{b}=(b_\sigma)_{\sigma\in G}$ in $F^G$, define ${\mathcal{F}}(\underline{b})=\underline{B}=(B_\chi)_{\chi\in \widehat{G}}$ in  $F^{\widehat{G}}$ by
$$
B_\chi:= \sum_{\sigma\in G} \chi (\sigma) b_\sigma \qquad (\chi\in \widehat{G}).
$$
Then ${\mathcal{F}}$ is a bijective map from $F^G$ to $F^{\widehat{G}}$, with  inverse ${\mathcal{F}}^{-1}$ defined by ${\mathcal{F}}^{-1}(\underline{B})=\underline{b}$ with
$$
b_\sigma=
\frac{1}{n} \sum_{\chi\in \widehat{G}} \chi (\sigma^{-1} ) B_\chi  \qquad (\sigma\in G).
$$
\end{proposition}

\begin{proof}
This follows from the  relation (see \cite{Se}, Prop.~7, \S2.5 and \cite{La} Chap.~XVIII, \S5 cor. 5.6)
$$
\frac{1}{n}  \sum_{\chi\in \widehat{G}} \chi (\sigma) \chi (\tau^{-1} ) =
\delta_{\sigma,\tau}
  $$
for $\sigma$ and $\tau$ in  $G$.
\end{proof}

For $\underline{b}\in F^G$,  let $M(\underline{b})$ be the $n\times n$ matrix
$$
M(\underline{b}):=
\left(
 b_{\tau^{-1}\sigma}
\right)_{\sigma,\tau\in G}.
$$
Then
$$
P^{-1} M(\underline{b}) P=
{\mathrm {Diag}}(
 B_\chi)_{\chi\in \widehat{G}},
$$
where  $P$ is the matrix   ($\ref{Equation:P}$) and   $\underline{B}= (B_\chi)_{\chi\in \widehat{G}}={\mathcal{F}}(\underline{b})$. 

For $\underline{B}\in F^{\widehat{G}}$,  let $\widehat{M}(\underline{B})$ be the $n\times n$ matrix
$$
\widehat{M}(\underline{B}):=
\left(
 B_{\psi^{-1}\chi}
\right)_{\chi,\psi\in \widehat{G}}.
$$
Applying the inverse transform ${\mathcal{F}}^{-1}$ given by Proposition $\ref{Proposition:InverseTransform}$ with ($\ref{Equation:P}'$), we deduce
\begin{equation}\label{Equation:PmoinsunMP}
{^t}\! P^{-1} \widehat{M}(\underline{B}) {^t}\! P=
n{\mathrm {Diag}}(
 b_{\sigma^{-1}})_{\sigma\in G}.
\end{equation}

\subsection{Case of a cyclic group: Finite Fourier Transform}
In the case of a cyclic group $G$, we recover the classical Finite Fourier Transform Pair
$$
B_h:= \sum_{\ell=0}^{n-1} \zeta^{h\ell} b_\ell,
\quad
b_\ell:= \frac{1}{n} \sum_{h=0}^{n-1} \zeta^{-h\ell} B_h,
$$
where, as before,  $\zeta$ is a primitive $n$--th root of unity.

\section{Hamming weight and Generalized Finite Fourier Transform}

A theorem of Blahut  \cite{JKK} relates the Hamming weight of a vector with the rank of a matrix defined by means of the Finite Fourier Transform. We extend it by replacing a cyclic group by an arbitrary finite abelian group $G$.

\begin{theorem} \label{thmBlahut}
The Hamming weight of $\underline{b}$ is the rank of the matrix $\widehat{M}(\underline{B})$ where $\underline{B}={\mathcal{F}}(\underline{b})$.
\end{theorem}

\begin{proof}
The rank of the  diagonal matrix in ($\ref{Equation:PmoinsunMP}$)  is the number of non--zero terms.
\end{proof}

\section{The matrix of a finite group }\label{S:MatrixFiniteGroup}

\subsection{An example:  ${\mathfrak S}_3$}

The symmetric group ${{\mathfrak S}}_3$ of order $6$ can be presented by generators and relations (with the unity element $e$), with the generators $\sigma$ and $\tau$ and the relations $\sigma^3=\tau^2=e$, $\tau\sigma\tau=\sigma^2$.

There exists an invertible $n\times n$ matrix $P$ such that
$$
P^{-1} A_{{\mathfrak S}_3} P=\begin{pmatrix}
L_0&0&0&0\\
0&L_1&0&0\\
0&0&M&0\\
0&0&0&M
\end{pmatrix},
$$
where $L_0$ and $L_1$ are the linear forms
$$
L_0=
X_e   + X_{\sigma} +X_{\sigma^2}  +  X_{\tau}+ X_{\tau\sigma} +  X_{\tau\sigma^2} ,
$$
$$
L_1:= X_e + X_{\sigma} +X_{\sigma^2}  -  X_{\tau} - X_{\tau\sigma} - X_{\tau\sigma^2}
$$
and $M$ is the $2\times 2$ matrix
$$
M=\begin{pmatrix}
X_e+jX_{\sigma}+j^2 X_{\sigma^2}
&
X_\tau+j^2X_{\tau\sigma}+j X_{\tau\sigma^2}
\\
X_\tau+jX_{\tau\sigma}+j^2 X_{\tau\sigma^2}
&
X_e+j^2X_{\sigma}+j X_{\sigma^2}
\end{pmatrix}.
$$
The linear forms $L_0$ and $L_1$ correspond to the representations of  ${\mathfrak S}_3$ of degree $1$, namely the trivial representation and the signature, while the matrix $M$ corresponds to the irreducible representation of degree $2$
(see \cite{Se} Chap.~5) defined by
$$
\sigma\mapsto  \begin{pmatrix}
j & 0
\\
0 & j^2
\end{pmatrix}
\quad\text{and}\quad
\tau\mapsto  \begin{pmatrix}
0&1
\\
1&0
\end{pmatrix}.
$$
Hence it also satisfies
$$
e\mapsto  \begin{pmatrix}
1 & 0
\\
0 & 1
\end{pmatrix},
\quad
\sigma^2\mapsto  \begin{pmatrix}
j ^2& 0
\\
0 & j
\end{pmatrix},
\quad
\tau\sigma\mapsto  \begin{pmatrix}
0&j^2
\\
j&0
\end{pmatrix},
\quad
\tau\sigma^2\mapsto  \begin{pmatrix}
0&j
\\
j^2&0
\end{pmatrix}.
$$
The determinant of $M$ is an irreducible polynomial in the ring
$$
{\mathbf C}[X_e, X_{\sigma}, X_{\sigma^2} ,   X_{\tau},  X_{\tau\sigma} ,   X_{\tau\sigma^2}].
$$
It can be written
$$
N(X_e, X_{\sigma}, X_{\sigma^2} )-N(X_{\tau},  X_{\tau\sigma} ,   X_{\tau\sigma^2}),
$$
where
$$
N(X_0,X_1,X_2)= {\mathrm {N}}_{{\mathbf Q}(j)/{\mathbf Q}}(X_0+jX_1+j^2X_2)
$$
(see ($\ref{Equation:troissurQ}$)). Cf. also \cite{Diaconis}.

 \subsection{The general case}

We assume again that the characteristic of the field $F$ does not divide the order $n$ of $G$ and that $F$ contains the primitive $n$--th roots of unity.  The regular representation   of $G$ has dimension $n$, its decomposition is well known (see for instance \cite{Se} Cor. 1 of Prop.~5 in \S2.4 and Chap.~5,  or \cite{La}, Chap.~VIII, \S4): each irreducible representation of $G$ is contained in the regular representation with a multiplicity equal to its degree $f$, so that the sum of the squares of  these degrees $f$ is $n$. Let $\varrho_1,\dots,\varrho_h$ be the irreducible representations and $f_1,\dots,f_h$ be their degrees. Hence there  is a basis of the space of the regular representation so that the associated matrix can be written as diagonal blocs
 $$
 {\mathrm {Diag}}(B_1,\dots,B_h),
 $$
 where, 
  for $1\le j\le h$, the matrix $B_j$ is a $f_j^2\times f_j^2$ matrix, which is a diagonal bloc of $f_j$ identical square matrices
 $$
 B_j={\mathrm {Diag}}(B_j^0,\dots,B_j^0),
 $$
 and $B_j^0$ is the $f_j\times f_j$ matrix associated with the representation $\varrho_j$. Using this change of bases and considering the generic element in the   group  ring $F[G]$,  one deduces that the matrix $A_G$ is equivalent to a matrix with the same shape, yielding a decomposition of the determinant into a product of polynomials
 $$
 \det A_G= \prod_{j=1}^h \Psi_{\varrho_j}^{f_j},
 $$
 where $\Psi_{\varrho_j}$ is a homogeneous polynomial of degree $f_j$. The fact that the representation $\varrho_j$ is irreducible implies that the polynomial $\Psi_{\varrho_j}$  is irreducible in $F[\underline{X}]$.

 \subsection{Frobenius}

 It is interesting, from a historical point of view, to look at the way Frobenius succeeded to produce the decomposition of the Gruppendeterminant into irreducible factors. The theory of linear representations of finite group was not yet fully developed: Frobenius was in the process of creating it. See  \cite{Bki}, historical note, and the references \cite{D,Fr1,Fr2}.

 Let $\varrho$ be an irreducible representation of a finite group $G$, $\chi$ its character, $f$ its degree. 
 Let us extend the map $\chi:G\mapsto{\mathbf C}$ into a 
 a function (again denoted by  $\chi$) on 
 $\bigcup_{k\ge 1} G^k$ with complex values by the induction formula, for $k\ge 1$, 
 $$
 \chi(s,s_1,\dots,s_k)=\chi(s)  \chi(s_1,\dots,s_k) -\sum_{i=1}^k  \chi(s_1,\dots,ss_i,\dots, s_k).
 $$
 For instance
 $$
 \chi(s_1,s_2)=\chi(s_1)\chi(s_2)-\chi(s_1s_2),
 $$
 $$
 \chi(e,s_1,\dots,s_k)=(f-k) \chi(s_1,\dots,s_k)
 $$
 and
 $$
  \chi(s_1,\dots,s_k) =0 \quad\text{for}\quad k>f.
  $$
    Define
  $$
  \Psi_\varrho:=(-1)^{f}
  \sum_{(s_1,\dots,s_{f})\in G^{f}  }\chi(s_1,\dots,s_{f})X_{s_1}\cdots X_{s_f} \in {\mathbf C}[\underline{X}].
  $$
  This is a homogeneous polynomial of degree $f$.

  An equivalent definition of  $\Psi_\varrho$ is the following. Let $A$ be the set of elements $(a_1,\dots,a_f)$ in ${\mathbf Z}^{f}$ satisfying
  $$
  a_j\ge 0\quad\hbox{for} \quad 1\le j\le f  \quad\hbox{and}\quad  \sum_{i=1}^f  ia_i=f.
  $$
  For $1\le k\le f$, set
  $$
  S_k=  \sum_{(s_1,\dots,s_k )\in G^{k}  }\chi(s_1\cdots s_k)X_{s_1}\cdots X_{s_k} .
  $$
    Then
  $$
  \Psi_\varrho
  =(-1)^{f}
  \sum_{(a_1,\dots,a_f) \in A} \prod_{k=1}^f \frac{S_k^{a_k} }{(-k)^{a_k} a_k!}\cdotp
  $$

  \begin{proposition}[Frobenius]
  The polynomial $  \Psi_\varrho$  is irreducible.  If $\varrho_1,\dots,\varrho_h$ are the irreducible representations of $G$ with degrees $f_1,\dots,f_h$ respectively, then
 $$
 \det A_G= \prod_{i=1}^h \Psi_{\varrho_i}^{f_i}
 $$
 is the decomposition of the polynomial $ \det A_G$ into irreducible factors in $F[\underline{X}]$.
  \end{proposition}

\section*{Acknowledgment}
The second author is thankful to the first author for a stay at Suda Neu-Tech Institute (Sanmenxia, Henan, China) where this joint work started, to the Abdus Salam School of Mathematical Science of Lahore where he pursued this work, to Jorge Jimenez Urroz for his help with Maple and to Claude Levesque for useful comments.

\bigskip\bigskip\bigskip

\noindent
{\sc Shigeru Kanemitsu}
\\
  Dept of Information Science,
\\
Faculty of Humanity-Oriented Science and Engineering
\\
Kinki University,
Iizuka,
FUKUOKA, 
820 8555 
JAPAN
\\
{\tt kanemitu@fuk.kindai.ac.jp}

\bigskip\bigskip\bigskip
\noindent
{\sc Michel~Waldschmidt} \\
{Universit\'e  Pierre et Marie Curie (Paris 6) }\\
{Institut de Math\'ematiques de Jussieu} \\
{4 Place Jussieu, 75252 PARIS Cedex 05, FRANCE} \\
{\tt miw@math.jussieu.fr}

\vfill
\vfill

\end{document}